\documentclass[12pt]{article}

\usepackage{amsmath, epsfig, cite}
\usepackage{amssymb}
\usepackage{amsfonts}
\usepackage{latexsym}
\usepackage{amsthm}

\newtheorem{thm}{Theorem}[section]
\newtheorem{prop}[thm]{Proposition}

\newtheorem{cor}[thm]{Corollary}

\numberwithin{equation}{section}

\renewcommand{\thefootnote}{}

\begin{document}

\begin{center}
{\large\bf $q$-Supercongruences from Jackson's $_8\phi_7$ summation
and Watson's $_8\phi_7$ transformation
 \footnote{The work is supported by the National Natural Science Foundation of China (No. 12071103).}}
\end{center}

\renewcommand{\thefootnote}{$\dagger$}

\vskip 2mm \centerline{Chuanan Wei}
\begin{center}
{School of Biomedical Information and Engineering,\\ Hainan Medical University, Haikou 571199, China
\\Email address: weichuanan78@163.com}
\end{center}


\vskip 0.7cm \noindent{\bf Abstract.} $q$-Supercongruences modulo
 the fifth and sixth powers of a cyclotomic polynomial are very rare in the literature.
 In this paper, we establish
some $q$-supercongruences modulo the fifth and sixth powers of a
cyclotomic polynomial in terms of Jackson's $_8\phi_7$ summation,
Watson's $_8\phi_7$ transformation, the creative microscoping method
recently introduced by Guo and Zudilin, and the Chinese remainder
theorem for coprime polynomials. More concretely, we give a
$q$-analogue of a nice formula due to Long and Ramakrishna [Adv.
Math. 290 (2016), 773--808] and two $q$-supercongruences involving
double series.



\section{Introduction}
For any nonnegative integer $n$ and complex number $x$, define the
shifted-factorial to be
\[(x)_{n}=\Gamma(x+n)/\Gamma(x),\]
where $\Gamma(x)$ is the famous Gamma function. Let $p$ be an odd
prime throughout the paper and $\mathbb{Z}_p$ denote the ring of all
$p$-adic integers. Define Morita's $p$-adic Gamma function (cf.
\cite[Chapter 7]{Robert}) by
 \[\Gamma_{p}(0)=1\quad \text{and}\quad \Gamma_{p}(n)
=(-1)^n\prod_{\substack{1\leqslant k< n\\
p\nmid k}}k,\quad \text{when}\quad n\in\mathbb{Z}^{+}.\] Noting that
$\mathbb{N}$ is a dense subset of $\mathbb{Z}_p$ associated with the
$p$-adic norm $|\cdot|_p$, for each $x\in\mathbb{Z}_p$, the
definition of $p$-adic Gamma function can be extended as
 \[\Gamma_{p}(x)
=\lim_{\substack{n\in\mathbb{N}\\
|x-n|_p\to0}}\Gamma_{p}(n).\]

 In 1997, Van Hamme
\cite[(D. 2)]{Hamme} conjectured the following supercongruence: for
$p\equiv 1\pmod 6$,
\begin{equation}\label{vanhamme}
\sum_{k=0}^{(p-1)/3}(6k+1)\frac{(1/3)_k^6}{k!^6}\equiv
 -p\Gamma_p(1/3)^9 \pmod{p^4}.
\end{equation}
 In
2016, Long and Ramakrishna \cite[Theorem 2]{LR} obtained the
generalization of \eqref{vanhamme}:
\begin{equation}\label{eq:long}
\sum_{k=0}^{p-1}(6k+1)\frac{(1/3)_k^6}{k!^6}\equiv
\begin{cases}\displaystyle -p\Gamma_p(1/3)^9  \pmod{p^6}, &\text{if $p\equiv
1\pmod 6$,}
\\[7pt]
\displaystyle -\frac{10}{27}p^4\Gamma_p(1/3)^9\pmod{p^6}, &\text{if
$p\equiv 5\pmod 6$.}
\end{cases}
\end{equation}
 Some results
and conjectures related to \eqref{eq:long} can be seen in Guo, Liu,
and Schlosser \cite{GLS}.

 For any complex numbers $x$ and $q$, define the $q$-shifted factorial
 to be
 \begin{equation*}
(x;q)_{0}=1\quad\text{and}\quad
(x;q)_n=(1-x)(1-xq)\cdots(1-xq^{n-1})\quad \text{when}\quad
n\in\mathbb{Z}^{+}.
 \end{equation*}
For simplicity, we usually adopt the compact notation
\begin{equation*}
(x_1,x_2,\dots,x_m;q)_{n}=(x_1;q)_{n}(x_2;q)_{n}\cdots(x_m;q)_{n},
 \end{equation*}
where $m\in\mathbb{Z}^{+}$ and $n\in\mathbb{Z}^{+}\cup\{0\}.$ Let
$[r]$ be the $q$-integer $(1-q^r)/(1-q)$ and $\Phi_n(q)$ stand for
the $n$-th cyclotomic polynomial in $q$:
\begin{equation*}
\Phi_n(q)=\prod_{\substack{1\leqslant k\leqslant n\\
\gcd(k,n)=1}}(q-\zeta^k),
\end{equation*}
where $\zeta$ is an $n$-th primitive root of unity. Recently, Guo
and Schlosser \cite[Theorem 2.3]{GS} discovered a partial
$q$-analogue of \eqref{eq:long}: for any positive integer $n$,
\begin{equation}\label{Guo}
\sum_{k=0}^{n-1}[6k+1]\frac{(q;q^3)_k^6}{(q^3;q^3)_k^6}q^{3k}\equiv
\begin{cases} \displaystyle 0  \pmod{[n]}, &\text{if $n\equiv 1\pmod 3$,}\\[10pt]
\displaystyle  0 \pmod{[n]\Phi_n(q)}, &\text{if $n\equiv 2\pmod 3$.}
\end{cases}
\end{equation}
They also proposed the following two conjectures
(cf. \cite[Conjectures 4.2 and 5.11]{GS}):
\begin{align*}
&\sum_{k=0}^{n-1}[2dk+1]\frac{(aq,q/a,bq,q/b;q^d)_k(q;q^d)_k^2}{(aq^d,q^d/a,bq^d,q^d/b;q^d)_k(q^d;q^d)_k^2}q^{(2d-3)k}
 \\[5pt]
&\quad\equiv
\begin{cases} \displaystyle 0  \pmod{[n]\Phi_n(q)}, &\text{if $n\equiv -1\pmod d$,}\\[10pt]
\displaystyle 0 \pmod{[n]}, &\text{otherwise,}
\end{cases}
\end{align*}
where $n>0$, $d\geq3$ are integers with $\gcd(n,d)=1$, and
\begin{align*}
&\sum_{k=0}^{n-1}[2dk-1]\frac{(aq^{-1},q^{-1}/a,bq^{-1},q^{-1}/b;q^d)_k(q^{-1};q^d)_k^2}{(aq^d,q^d/a,bq^d,q^d/b;q^d)_k(q^d;q^d)_k^2}q^{(2d+3)k}
 \\[5pt]
&\quad\equiv
\begin{cases} \displaystyle 0  \pmod{[n]\Phi_n(q)}, &\text{if $n\equiv 1\pmod d$,}\\[10pt]
\displaystyle 0 \pmod{[n]}, &\text{otherwise,}
\end{cases}
\end{align*}
where $n>1$, $d\geq3$ are integers satisfying $\gcd(n,d)=1$. It
should be pointed out that the two conjectures have been proved by
Ni and Wang \cite{NW}. There are more $q$-analogues of
supercongruences in the literature, we refer the reader to
\cite{Guo-R,GS20c,GuoZu-b,LW,LP,Tauraso,WY-a,Wei,Zu19}.

In this paper, we shall establish the following two theorems, which
extend \eqref{Guo} and can be regarded as the further $q$-analogue
of \eqref{eq:long}.

\begin{thm}\label{thm-b}
Let $n$ be a positive integer such that $n\equiv 1\pmod 3$. Then,
modulo $[n]\Phi_n(q)^4$,
\begin{align*}
\sum_{k=0}^{M}[6k+1]\frac{(q;q^3)_k^6}{(q^3;q^3)_k^6}q^{3k}
&\equiv[n]
\frac{(q^2;q^3)_{(n-1)/3}^3}{(q^3;q^3)_{(n-1)/3}^3}\\[2mm]
&\quad\times\:
\bigg\{1+[n]^2(2-q^{n})\sum_{j=1}^{(n-1)/3}\bigg(\frac{q^{3j-1}}{[3j-1]^2}-\frac{q^{3j}}{[3j]^2}\bigg)\bigg\},
\end{align*}
where $M=(n-1)/3$ or $n-1$.
\end{thm}

\begin{thm}\label{thm-c}
Let $n$ be a positive integer such that $n\equiv 2\pmod 3$. Then,
modulo $[n]\Phi_n(q)^5$,
\begin{align*}
\sum_{k=0}^{M}[6k+1]\frac{(q;q^3)_k^6}{(q^3;q^3)_k^6}q^{3k} &\equiv
5[2n] \frac{(q^2;q^3)_{(2n-1)/3}^3}{(q^3;q^3)_{(2n-1)/3}^3},
\end{align*}
where $M=(2n-1)/3$ or $n-1$.
\end{thm}

Fixing $n=p^s$ and taking $q\to1$ in Theorem \ref{thm-b}, we arrive
at the conclusion.

\begin{cor}\label{cor-b}
Let $p$ be an odd prime and $s$ a positive integer such that
$p^s\equiv 1\pmod{3}$. Then, modulo $p^{s+4}$,
\begin{align*}
\sum_{k=0}^{m}(6k+1)\frac{(1/3)_k^6}{k!^6} \equiv
\frac{(2/3)_{(p^s-1)/3}^3}{(1)_{(p^s-1)/3}^3}\bigg\{p^s+p^{3s}\sum_{j=1}^{(p^s-1)/3}\bigg(\frac{1}{(3j-1)^2}-\frac{1}{(3j)^2}\bigg)\bigg\},
\end{align*}
where $m=(p^s-1)/3$ or $p^s-1$.
\end{cor}

Setting $n=p^s$ and taking $q\to 1$ in Theorem \ref{thm-c}, we are
led to the formula.
\begin{cor}\label{cor-c}
Let $p$ be an odd prime and $s$ a positive integer such that
$p^s\equiv 2\pmod{3}$. Then, modulo $p^{s+5}$,
\begin{align*}
\sum_{k=0}^{m}(6k+1)\frac{(1/3)_k^6}{k!^6} \equiv
10p^s\frac{(2/3)_{(2p^s-1)/3}^3}{(1)_{(2p^s-1)/3}^3},
\end{align*}
where $m=(2p^s-1)/3$ or $p^s-1$.
\end{cor}

According to the $s=1$ case of Corollaries \ref{cor-b} and
\ref{cor-c} and \eqref{eq:long}, we arrive at the following two
relations.

\begin{cor}\label{coro}
Let $p$ be an odd prime. If $p\equiv 1\pmod 6$, then
\begin{align*}
\frac{(2/3)_{(p-1)/3}^3}{(1)_{(p-1)/3}^3}\bigg\{1+p^2\sum_{j=1}^{(p-1)/3}\bigg(\frac{1}{(3j-1)^2}-\frac{1}{(3j)^2}\bigg)\bigg\}
\equiv -\Gamma_p(1/3)^9\pmod{p^4}.
\end{align*}
If $p\equiv 5\pmod 6$, then
\begin{align*}
\frac{(2/3)_{(2p-1)/3}^3}{(1)_{(2p-1)/3}^3} \equiv
-\frac{p^3}{27}\Gamma_p(1/3)^9\pmod{p^5}.
\end{align*}
\end{cor}

The rest of the paper is arranged as follows. Via the creative
microscoping method from Guo and Zudilin \cite{GuoZu-a} and the
Chinese remainder theorem for coprime polynomials, we shall derive
the parametric generalizations of Theorems \ref{thm-b} and
\ref{thm-c} and then prove the two theorems in Section 2.
Furthermore, two $q$-supercongruences involving double series will
be established in Sections 3.

\section{Proof of Theorems \ref{thm-b} and \ref{thm-c}}
In order to prove Theorems \ref{thm-b} and \ref{thm-c}, we need the
following proposition.

\begin{prop}\label{prop-aa}
Let $n$ be a positive integer such that $n\equiv t\pmod 3$. Then,
modulo $(1-aq^{tn})(a-q^{tn})(1-bq^{tn})(b-q^{tn})$,
\begin{align}
&\sum_{k=0}^{T}[6k+1]\frac{(aq,q/a,bq,q/b;q^3)_k(q;q^3)_k^2}{(q^3/a,aq^3,q^{3}/b,bq^3;q^3)_k(q^3;q^3)_k^2}q^{3k}
\notag\\[5pt]
&\quad\equiv\,[tn]\frac{(1-bq^{tn})(b-q^{tn})(-1-a^2+aq^{tn})}{(a-b)(1-ab)}\frac{(bq^2,q^2/b,q^2;q^3)_{(tn-1)/3}}{(q^3/b,bq^3,q^3;q^3)_{(tn-1)/3}}
\notag\\[5pt]
&\quad\quad+\:[tn]\frac{(1-aq^{tn})(a-q^{tn})(-1-b^2+bq^{tn})}{(b-a)(1-ba)}\frac{(aq^2,q^2/a,q^2;q^3)_{(tn-1)/3}}{(q^3/a,aq^3,q^3;q^3)_{(tn-1)/3}},
\label{eq:wei-cc}
\end{align}
where $T=(tn-1)/3$ or $n-1$ and $t\in\{1,2\}$.
\end{prop}

\begin{proof}
When $a=q^{-tn}$ or $a=q^{tn}$, the left-hand side of
\eqref{eq:wei-cc} is equal to
\begin{align}
&\sum_{k=0}^{T}[6k+1]\frac{(q^{1-tn},q^{1+tn},bq,q/b;q^3)_k(q;q^3)_k^2}{(q^{3+tn},q^{3-tn},q^{3}/b,bq^3;q^3)_k(q^3;q^3)_k^2}q^{3k}
\notag\\[5pt]
&= {_8\phi_7}\!\left[\begin{array}{cccccccc} q,& q^{\frac{7}{2}},&
-q^{\frac{7}{2}}, & q,  & bq,   & q/b, & q^{1+tn}, & q^{1-tn}
\\[5pt]
  & q^{\frac{1}{2}}, & -q^{\frac{1}{2}}, & q^{3}, & q^3/b,  & bq^3, & q^{3-tn}, & q^{3+tn}
\end{array};q^3,\, q^3
\right],
 \label{eq:wei-dd}
\end{align}
where the basic hypergeometric series (cf. \cite{Gasper}) has been
defined by
$$
_{r+1}\phi_{r}\left[\begin{array}{c}
a_0,a_1\ldots,a_{r}\\
\quad\:\: b_1,\ldots,b_{r}
\end{array};q,\, z
\right] =\sum_{k=0}^{\infty}\frac{(a_0,a_1,\ldots, a_{r};q)_k}
{(q,b_1,\ldots,b_{r};q)_k}z^k.
$$
Via Jackon's $_8\phi_7$ summation (cf. \cite[Appendix
(II.22)]{Gasper}):
\begin{align*}
& _{8}\phi_{7}\!\left[\begin{array}{cccccccc}
a,& qa^{\frac{1}{2}},& -qa^{\frac{1}{2}}, & b,    & c,    & d,    & e,    & q^{-n} \\
  & a^{\frac{1}{2}}, & -a^{\frac{1}{2}},  & aq/b, & aq/c, & aq/d, & aq/e, & aq^{n+1}
\end{array};q,q
\right] \notag\\[5pt]
&\quad =\frac{(aq,aq/bc,aq/bd,aq/cd;q)_{n}}{(aq/b,aq/c,aq/d,aq/bcd;q)_{n}},
\end{align*}
where $a^2q=bcdeq^{-n}$, the right-hand side of \eqref{eq:wei-dd}
can be stated as
\begin{align*}
[tn]\frac{(bq^2,q^2/b,q^2;q^3)_{(tn-1)/3}}{(q^3/b,bq^3,q^3;q^3)_{(tn-1)/3}}.
\end{align*}
Because $(1-aq^{tn})$ and $(a-q^{tn})$ are relatively prime
polynomials, we get the following result: modulo
$(1-aq^{tn})(a-q^{tn})$,
\begin{align}
\sum_{k=0}^{T}[6k+1]\frac{(aq,q/a,bq,q/b;q^3)_k(q;q^3)_k^2}{(q^3/a,aq^3,q^{3}/b,bq^3;q^3)_k(q^3;q^3)_k^2}q^{3k}
\equiv\,[tn]\frac{(bq^2,q^2/b,q^2;q^3)_{(tn-1)/3}}{(q^3/b,bq^3,q^3;q^3)_{(tn-1)/3}}.
\label{eq:wei-ee}
\end{align}
Interchanging the parameters $a$ and $b$ in \eqref{eq:wei-ee}, we
have the $q$-supercongruence:  modulo  $(1-bq^{tn})(b-q^{tn})$,
\begin{align}
\sum_{k=0}^{T}[6k+1]\frac{(aq,q/a,bq,q/b;q^3)_k(q;q^3)_k^2}{(q^3/a,aq^3,q^{3}/b,bq^3;q^3)_k(q^3;q^3)_k^2}q^{3k}
\equiv\,[tn]\frac{(aq^2,q^2/a,q^2;q^3)_{(tn-1)/3}}{(q^3/a,aq^3,q^3;q^3)_{(tn-1)/3}}.
\label{eq:wei-ff}
\end{align}

It is clear that the polynomials $(1-aq^{tn})(a-q^{tn})$ and
$(1-bq^{tn})(b-q^{tn})$ are relatively prime.
 Noting the relations:
\begin{align}
&\frac{(1-bq^{tn})(b-q^{tn})(-1-a^2+aq^{tn})}{(a-b)(1-ab)}\equiv1\pmod{(1-aq^{tn})(a-q^{tn})},
\label{relation-aa}
\\[5pt]
&\frac{(1-aq^{tn})(a-q^{tn})(-1-b^2+bq^{tn})}{(b-a)(1-ba)}\equiv1\pmod{(1-bq^{tn})(b-q^{tn})}
\label{relation-bb}
\end{align}
and employing the Chinese remainder theorem for coprime polynomials,
from \eqref{eq:wei-ee} and \eqref{eq:wei-ff} we can deduce the
 $q$-supercongruence \eqref{eq:wei-cc}.
\end{proof}

 Now we are ready to give the parametric generalizations of Theorems
\ref{thm-b} and \ref{thm-c}.

\begin{thm}\label{thm-aa}
Let $n$ be a positive integer such that $n\equiv 1\pmod 3$. Then,
modulo $[n](1-aq^{n})(a-q^{n})(1-bq^{n})(b-q^{n})$,
\begin{align*}
&\sum_{k=0}^{M}[6k+1]\frac{(aq,q/a,bq,q/b;q^3)_k(q;q^3)_k^2}{(q^3/a,aq^3,q^{3}/b,bq^3;q^3)_k(q^3;q^3)_k^2}q^{3k}
\notag\\[5pt]
&\quad\equiv\,[n]\frac{(1-bq^{n})(b-q^{n})(-1-a^2+aq^{n})}{(a-b)(1-ab)}\frac{(bq^2,q^2/b,q^2;q^3)_{(n-1)/3}}{(q^3/b,bq^3,q^3;q^3)_{(n-1)/3}}
\notag\\[5pt]
&\quad\quad+\:[n]\frac{(1-aq^{n})(a-q^{n})(-1-b^2+bq^{n})}{(b-a)(1-ba)}\frac{(aq^2,q^2/a,q^2;q^3)_{(n-1)/3}}{(q^3/a,aq^3,q^3;q^3)_{(n-1)/3}},
\end{align*}
where $M=(n-1)/3$ or $n-1$.
\end{thm}

\begin{proof}
Ni and Wang \cite[Lemma 2.2]{NW}) provides
\begin{align}
&\sum_{k=0}^{m}[2dk+r]\frac{(aq^r,q^r/a,bq^r,q^r/b,q^r/c,q^r;q^d)_k}{(q^d/a,aq^d,q^d/b,bq^d,cq^d,q^d;q^d)_k}(cq^{2d-3r})^k
\equiv 0\pmod{[n]}, \label{eq:NW-a}
\\[5pt]
&\sum_{k=0}^{n-1}[2dk+r]\frac{(aq^r,q^r/a,bq^r,q^r/b,q^r/c,q^r;q^d)_k}{(q^d/a,aq^d,q^d/b,bq^d,cq^d,q^d;q^d)_k}(cq^{2d-3r})^k
\equiv 0\pmod{[n]}, \label{eq:NW-b}
\end{align}
where $n,d$ are positive integers and $r$ is an integer such that
$0\leq m\leq n-1$, $\gcd(n,d)=1$, and $dm\equiv -r\pmod{n}$.

 Letting
$c=1$, $d=3$, $\mu=(n-1)/3$, $r=1$ in \eqref{eq:NW-a} and
\eqref{eq:NW-b}, we have
\begin{align}
\sum_{k=0}^{M}[6k+1]\frac{(aq,q/a,bq,q/b;q^3)_k(q;q^3)_k^2}{(q^3/a,aq^3,q^{3}/b,bq^3;q^d)_k(q^3;q^3)_k^2}q^{3k}
\equiv 0\pmod{[n]},\label{eq:wei-hh}
\end{align}
where $n$ is a positive integer satisfying $n\equiv1\pmod{3}$.

Since $(1-aq^{n})(a-q^{n})(1-bq^{n})(b-q^{n})$ and $[n]$ are
relatively prime polynomials, we can find Theorem \ref{thm-aa} by
the $t=1$ case of Proposition \ref{prop-aa} and \eqref{eq:wei-hh}.
\end{proof}

\begin{thm}\label{thm-bb}
Let $n$ be a positive integer such that $n\equiv 2\pmod 3$. Then,
modulo $[n]\Phi_n(q)(1-aq^{2n})(a-q^{2n})(1-bq^{2n})(b-q^{2n})$,
\begin{align}
&\sum_{k=0}^{M}[6k+1]\frac{(aq,q/a,bq,q/b;q^3)_k(q;q^3)_k^2}{(q^3/a,aq^3,q^{3}/b,bq^3;q^3)_k(q^3;q^3)_k^2}q^{3k}
\notag\\[5pt]
&\quad\equiv\,[2n]\frac{(1-bq^{2n})(b-q^{2n})(-1-a^2+aq^{2n})}{(a-b)(1-ab)}\frac{(bq^2,q^2/b,q^2;q^3)_{(2n-1)/3}}{(q^3/b,bq^3,q^3;q^3)_{(2n-1)/3}}
\notag\\[5pt]
&\quad\quad+\:[2n]\frac{(1-aq^{2n})(a-q^{2n})(-1-b^2+bq^{2n})}{(b-a)(1-ba)}\frac{(aq^2,q^2/a,q^2;q^3)_{(2n-1)/3}}{(q^3/a,aq^3,q^3;q^3)_{(2n-1)/3}},
\label{eq:wei-ii}
\end{align}
 where $M=(2n-1)/3$ or $n-1$.
\end{thm}

\begin{proof}
A known result due to Ni and Wang \cite[Theorem 2.3]{NW}) is
\begin{align}
\sum_{k=0}^{\nu}[2dk+r]\frac{(aq^r,q^r/a,bq^r,q^r/b;q^d)_k(q^r;q^d)_k^2}{(q^d/a,aq^d,q^d/b,bq^d;q^d)_k(q^d;q^d)_k^2}q^{(2d-3r)k}
\equiv 0\pmod{[n]\Phi_n(q)}, \label{eq:NW}
\end{align}
where $n>1$, $d\geq3$ are integers, $r=\pm1$, and $\nu=(dn-n-r)/d$
or $n-1$ such that $n\geq d-r$, $\gcd(n,d)=1$, and
$n\equiv-r\pmod{d}$. Letting $d=3$, $r=1$ in \eqref{eq:NW}, we
arrive at
\begin{align}
\sum_{k=0}^{M}[6k+1]\frac{(aq,q/a,bq,q/b;q^3)_k(q;q^3)_k^2}{(q^3/a,aq^3,q^{3}/b,bq^3;q^d)_k(q^3;q^3)_k^2}q^{3k}
\equiv 0\pmod{[n]\Phi_n(q)},\label{eq:wei-jj}
\end{align}
where $n$ is a positive integer with $n\equiv2\pmod{3}$. According
to the method, which is used to prove Guo \cite[Lemma 1]{Guo-R}, and
noting that the factor $(1-q^n)$ appears in $(q^2;q^3)_{(2n-1)/3}$,
it is routine to see that
\begin{align*}
[2n]\frac{(q^2;q^3)_{(2n-1)/3}}{(q^3;q^3)_{(2n-1)/3}} \equiv
0\pmod{[n]\Phi_n(q)}.
\end{align*}
So we prove that \eqref{eq:wei-ii} is correct modulo $[n]\Phi_n(q)$.
Some similar discuss will be omitted elsewhere in the paper.

Because $(1-aq^{2n})(a-q^{2n})(1-bq^{2n})(b-q^{2n})$ and
$[n]\Phi_n(q)$ are relatively prime polynomials, we can establish
\eqref{eq:wei-ii} by the $t=2$ case of Proposition \ref{prop-aa} and
the upward conclusion.
\end{proof}

Subsequently, we shall display the proof of Theorems \ref{thm-b} and
\ref{thm-c}.

\begin{proof}[Proof of Theorem \ref{thm-b}]
The $b\to1$ case of Theorem \ref{thm-aa} yields the formula: modulo
$[n]\Phi_n(q)^2(1-aq^{n})(a-q^{n})$,
\begin{align}
&\sum_{k=0}^{M}[6k+1]\frac{(aq,q/a;q^3)_k(q;q^3)_k^4}{(q^3/a,aq^3;q^d)_k(q^3;q^3)_k^4}q^{3k}
\notag\\[5pt]
&\quad\equiv\,[n]\frac{(1-q^{n})^2(1+a^2-aq^{n})}{(1-a)^2}\frac{(q^2;q^3)_{(n-1)/3}^3}{(q^3;q^3)_{(n-1)/3}^3}
\notag\\[5pt]
&\quad\quad-\:[n]\frac{(1-aq^{n})(a-q^{n})(2-q^{n})}{(1-a)^2}\frac{(aq^2,q^2/a,q^2;q^3)_{(n-1)/3}}{(q^3/a,aq^3,q^3;q^3)_{(n-1)/3}}
\notag\\[5pt]
&\quad\equiv[n](1-q^{n})^2\frac{(q^2;q^3)_{(n-1)/3}^3}{(q^3;q^3)_{(n-1)/3}^3}
\notag\\[5pt]
&\quad\quad+[n]\frac{a(1-q^{n})^2(2-q^{n})}{(1-a)^2}\frac{(q^2;q^3)_{(n-1)/3}^3}{(q^3;q^3)_{(n-1)/3}^3}
\notag\\[5pt]
&\quad\quad-\:[n]\frac{(1-aq^{n})(a-q^{n})(2-q^{n})}{(1-a)^2}\frac{(aq^2,q^2/a,q^2;q^3)_{(n-1)/3}}{(q^3/a,aq^3,q^3;q^3)_{(n-1)/3}}.
\label{eq:wei-kk}
\end{align}
 By the L'H\^{o}pital rule, we have
\begin{align*}
&\lim_{a\to1}\bigg\{\frac{a(1-q^{n})^2}{(1-a)^2}\frac{(q^2;q^3)_{(n-1)/3}^2}{(q^3;q^3)_{(n-1)/3}^2}
-\frac{(1-aq^{n})(a-q^{n})}{(1-a)^2}\frac{(aq^2,q^2/a;q^3)_{(n-1)/3}}{(q^3/a,aq^3;q^3)_{(n-1)/3}}\bigg\}\\[5pt]
&\quad=\frac{(q^2;q^3)_{(n-1)/3}^2}{(q^3;q^3)_{(n-1)/3}^2}
\bigg\{q^{n}+[n]^2\sum_{j=1}^{(n-1)/3}\bigg(\frac{q^{3j-1}}{[3j-1]^2}-\frac{q^{3j}}{[3j]^2}\bigg)\bigg\}.
\end{align*}
Letting $a\to1$ in \eqref{eq:wei-kk} and employing the above limit,
we obtain Theorem \ref{thm-b}.
\end{proof}

\begin{proof}[Proof of Theorem \ref{thm-c}]
The $b\to1$ case of Theorem \ref{thm-bb} results in the conclusion:
modulo $[n]\Phi_n(q)^3(1-aq^{2n})(a-q^{2n})$,
\begin{align}
&\sum_{k=0}^{M}[6k+1]\frac{(aq,q/a;q^3)_k(q;q^3)_k^4}{(q^3/a,aq^3;q^d)_k(q^3;q^3)_k^4}q^{3k}
\notag\\[5pt]
&\quad=\,[2n](1-q^{2n})^2\frac{(q^2;q^3)_{(2n-1)/3}^3}{(q^3;q^3)_{(2n-1)/3}^3}
\notag\\[5pt]
&\quad\quad+[2n]\frac{a(1-q^{2n})^2(2-q^{2n})}{(1-a)^2}\frac{(q^2;q^3)_{(2n-1)/3}^3}{(q^3;q^3)_{(2n-1)/3}^3}
\notag\\[5pt]
&\quad\quad-\:[2n]\frac{(1-aq^{2n})(a-q^{2n})(2-q^{2n})}{(1-a)^2}\frac{(aq^2,q^2/a,q^2;q^3)_{(2n-1)/3}}{(q^3/a,aq^3,q^3;q^3)_{(2n-1)/3}}.
\label{eq:wei-ll}
\end{align}
 By the L'H\^{o}pital rule, we have
\begin{align*}
&\lim_{a\to1}\bigg\{\frac{a(1-q^{2n})^2}{(1-a)^2}\frac{(q^2;q^3)_{(2n-1)/3}^2}{(q^3;q^3)_{(2n-1)/3}^2}
-\frac{(1-aq^{2n})(a-q^{2n})}{(1-a)^2}\frac{(aq^2,q^2/a;q^3)_{(2n-1)/3}}{(q^3/a,aq^3;q^3)_{(2n-1)/3}}\bigg\}\\[5pt]
&\quad=\frac{(q^2;q^3)_{(2n-1)/3}^2}{(q^3;q^3)_{(2n-1)/3}^2}
\bigg\{q^{2n}+[2n]^2\sum_{j=1}^{(2n-1)/3}\bigg(\frac{q^{3j-1}}{[3j-1]^2}-\frac{q^{3j}}{[3j]^2}\bigg)\bigg\}.
\end{align*}
Letting $a\to1$ in \eqref{eq:wei-ll} and utilizing the upper limit,
we get the $q$-supercongruence: modulo $[n]\Phi_n(q)^5$,
\begin{align}
\sum_{k=0}^{M}[6k+1]\frac{(q;q^3)_k^6}{(q^3;q^3)_k^6}q^{3k}
&\equiv[2n] \frac{(q^2;q^3)_{(2n-1)/3}^3}{(q^3;q^3)_{(2n-1)/3}^3}
 \notag\\[5pt]
&\quad\times\:
\bigg\{1+[2n]^2(2-q^{2n})\sum_{j=1}^{(2n-1)/3}\bigg(\frac{q^{3j-1}}{[3j-1]^2}-\frac{q^{3j}}{[3j]^2}\bigg)\bigg\}.
\label{eq:wei-oo}
\end{align}
It is ordinary to certify the relation:
\begin{align}
&1+[2n]^2(2-q^{2n})\sum_{j=1}^{(2n-1)/3}\bigg(\frac{q^{3j-1}}{[3j-1]^2}-\frac{q^{3j}}{[3j]^2}\bigg)
\notag\\[4pt]
&\quad \equiv 1+[2n]^2(2-q^{2n})\frac{q^n}{[n]^2}
\notag\\[4pt]
&\quad \equiv5\pmod{\Phi_n(q)^2}. \label{eq:wei-pp}
\end{align}
Considering that the factor $(1-q^n)$ appears in
$(q^2;q^3)_{(2n-1)/3}$, the combination of \eqref{eq:wei-oo} and
\eqref{eq:wei-pp} produces Theorem \ref{thm-c}.
\end{proof}

\section{$q$-Supercongruences involving
double series.}
The main results of this section are the following two theorems.

\begin{thm}\label{thm-d}
Let $n$, $d$ be positive integers and $r$ an integer such that
$d+n-dn\leq r\leq n$, $\gcd(n,d)=1$, and $n\equiv r\pmod{d}$. Then,
modulo $[n]\Phi_n(q)^4$,
\begin{align*}
&\sum_{k=0}^{M}[2dk+r]\frac{(q^r;q^d)_k^5(cq^r;q^d)_k}{(q^d;q^d)_k^5(q^d/c;q^d)_k}\bigg(\frac{q^{2d-3r}}{c}\bigg)^k\\[2mm]
&\quad\equiv[n](cq^r)^{(r-n)/d}
\frac{(cq^{2r};q^d)_{(n-r)/d}}{(q^d/c;q^d)_{(n-r)/d}}
\\[2mm]
&\quad\quad\times\:
\sum_{k=0}^{(n-r)/d}\frac{(q^r;q^d)_k^2(q^{d-r},cq^r;q^d)_kq^{dk}}{(q^d;q^d)_k^3(cq^{2r};q^d)_k}
\\[2mm]
&\quad\quad\times\:
\bigg\{1-[n]^2(2-q^{n})\sum_{j=1}^{k}\bigg(\frac{q^{dj}}{[dj]^2}+\frac{q^{dj-d+r}}{[dj-d+r]^2}\bigg)\bigg\},
\end{align*}
where $M=(n-r)/d$ or $n-1$.
\end{thm}

\begin{thm}\label{thm-e}
Let $n$, $d$ be integers such that $n+r\geq d\geq 3$, $\gcd(n,d)=1$,
and $n\equiv-r\pmod{d}$. Then, modulo $[n]\Phi_n(q)^5$,
\begin{align*}
&\sum_{k=0}^{M}[2dk+r]\frac{(q^r;q^d)_k^6}{(q^d;q^d)_k^6}q^{(2d-3r)k}\\[5pt]
&\quad\equiv[dn-n]q^{r(r+n-dn)/d}
\frac{(q^{2r};q^d)_{(dn-n-r)/d}}{(q^d;q^d)_{(dn-n-r)/d}}\\[5pt]
&\qquad\times\:
\sum_{k=0}^{(dn-n-r)/d}\frac{(q^r;q^d)_k^3(q^{d-r};q^d)_kq^{dk}}{(q^d;q^d)_k^3(q^{2r};q^d)_k}
\end{align*}
\begin{align*}
&\qquad\times\:
\bigg\{1-[dn-n]^2(2-q^{dn-n})\sum_{j=1}^{k}\bigg(\frac{q^{dj}}{[dj]^2}+\frac{q^{dj-d+r}}{[dj-d+r]^2}\bigg)\bigg\},
\end{align*}
where $r=\pm1$  and $M=(dn-n-r)/d$ or $n-1$.
\end{thm}

Theorems \ref{thm-d} and \ref{thm-e} are respectively the
generalizations of Theorems \ref{thm-b} and \ref{thm-c}. Some
special cases from them may be displayed as follows.

Choosing $n=p^s$ and taking $c\to1, q\to 1$ in Theorem \ref{thm-d},
we obtain the conclusion.
\begin{cor}\label{cor-e}
Let $d$, $s$ be positive integers, $p$ an odd prime,
 and $r$ an integer such that $d+p^s-dp^s\leq r\leq p^s$, $\gcd(p,d)=1$, and $p^s\equiv r\pmod{d}$. Then, modulo $p^{s+4}$,
\begin{align*}
&\sum_{k=0}^{m}(2dk+r)\frac{(r/d)_k^6}{k!^6}\equiv\frac{(2r/d)_{(p^s-r)/d}}{(1)_{(p^s-r)/d}}
\\[5pt]
&\:\:\quad\times
\sum_{k=0}^{(p^s-r)/d}\frac{(r/d)_k^3(1-r/d)_k}{k!^3(2r/d)_k}
\bigg\{p^s-p^{3s}\sum_{j=1}^{k}\bigg(\frac{1}{(dj)^2}+\frac{1}{(dj-d+r)^2}\bigg)\bigg\},
\end{align*}
where $m=(p^s-r)/d$ or $p^s-1$.
\end{cor}

Fixing $n=p^s$ and taking $c\to-1, q\to 1$ in Theorem \ref{thm-d},
we get the supercongruence.

\begin{cor}\label{cor-g}
Let $d$, $s$ be positive integers, $p$ an odd prime,
 and $r$ an integer such that $d+p^s-dp^s\leq r\leq p^s$, $\gcd(p,d)=1$, and $p^s\equiv r\pmod{d}$. Then, modulo $p^{s+4}$,
\begin{align*}
&\sum_{k=0}^{m}(-1)^k(2dk+r)\frac{(r/d)_k^5}{k!^5}\equiv(-1)^{(r-p^s)/d}
\\[5pt]
&\:\:\quad\times
\sum_{k=0}^{(p^s-r)/d}\frac{(r/d)_k^2(1-r/d)_k}{k!^3}
\bigg\{p^s-p^{3s}\sum_{j=1}^{k}\bigg(\frac{1}{(dj)^2}+\frac{1}{(dj-d+r)^2}\bigg)\bigg\}.
\end{align*}
where $m=(p^s-r)/d$ or $p^s-1$.
\end{cor}

Setting $n=p^s$ and taking $q\to 1$ in Theorem \ref{thm-e}, we
arrive at the formula.
\begin{cor}\label{cor-h}
Let $d$, $s$ be positive integers and $p$ an odd prime such that
$p^s+r\geq d\geq 3$, $\gcd(p,d)=1$, and $p^s\equiv-r\pmod{d}$. Then,
modulo $p^{s+5}$,
\begin{align*}
&\sum_{k=0}^{m}(2dk+r)\frac{(r/d)_k^6}{k!^6}
\equiv\frac{(2r/d)_{(dp^s-p^s-r)/d}}{(1)_{(dp^s-p^s-r)/d}}
\sum_{k=0}^{(dp^s-p^s-r)/d}\frac{(r/d)_k^3(1-r/d)_k}{k!^3(2r/d)_k}\\[5pt]
&\quad\times\:
\bigg\{(d-1)p^s-(d-1)^3p^{3s}\sum_{j=1}^{k}\bigg(\frac{1}{(dj)^2}+\frac{1}{(dj-d+r)^2}\bigg)\bigg\},
\end{align*}
where $r=\pm1$ and $m=(dp^s-p^s-r)/d$ or $p^s-1$.
\end{cor}

To achieve the goal of proving Theorems \ref{thm-d} and \ref{thm-e},
we shall establish the following proposition above all.

\begin{prop}\label{prop-aaa}
Let $n$, $d$ be positive integers and $r$ an integer
 such that $d+tn-dn\leq r\leq tn$,
$\gcd(n,d)=1$, and $tn\equiv r\pmod{d}$. Then, modulo
$(1-aq^{tn})(a-q^{tn})(1-bq^{tn})(b-q^{tn})$,
\begin{align}
&\sum_{k=0}^{T}[2dk+r]\frac{(aq^r,q^r/a,bq^r,q^r/b,cq^r,q^r;q^d)_k}{(q^d/a,aq^d,q^{d}/b,bq^d,q^d/c,q^d;q^d)_k}\bigg(\frac{q^{2d-3r}}{c}\bigg)^k
\notag
\\[5pt]
&\quad\equiv\,[tn](cq^r)^{(r-tn)/d}\frac{(cq^{2r};q^d)_{(tn-r)/d}}{(q^d/c;q^d)_{(tn-1)/d}}
\notag\\[5pt]
&\quad\times\,\bigg\{\frac{(1-bq^{tn})(b-q^{tn})(-1-a^2+aq^{tn})}{(a-b)(1-ab)}
\sum_{k=0}^{(tn-r)/d}\frac{(aq^r,q^r/a,cq^r,q^{d-r};q^d)_k}{(bq^d,q^d/b,cq^{2r},q^d;q^d)_k}q^{dk}
\notag\\[5pt]
&\quad\quad+\:\frac{(1-aq^{tn})(a-q^{tn})(-1-b^2+bq^{tn})}{(b-a)(1-ba)}
\sum_{k=0}^{(tn-r)/d}\frac{(bq^r,q^r/b,cq^r,q^{d-r};q^d)_k}{(aq^d,q^d/a,cq^{2r},q^d;q^d)_k}q^{dk}\bigg\},
\label{eq:wei-aaa}
\end{align}
where $T=(tn-r)/d$ or $n-1$ and  $t\in\{1,d-1\}$.
\end{prop}

\begin{proof}
When $a=q^{-tn}$ or $a=q^{tn}$, the left-hand side of
\eqref{eq:wei-aaa} is equal to
\begin{align}
&\sum_{k=0}^{T}[2dk+r]\frac{(q^{r-tn},q^{r+tn},bq^r,q^r/b,cq^r,q^r;q^d)_k}{(q^{d+tn},q^{d-tn},q^{d}/b,bq^d,q^d/c,q^d;q^d)_k}\bigg(\frac{q^{2d-3r}}{c}\bigg)^k
\notag\\[5pt]
&= [r]{_8\phi_7}\!\left[\begin{array}{cccccccc} q^r,&
q^{d+\frac{r}{2}},& -q^{d+\frac{r}{2}}, & bq^r,   & q^r/b, & cq^r, &
q^{r+tn}, & q^{r-tn}
\\[5pt]
  & q^{\frac{r}{2}}, & -q^{\frac{r}{2}}, & q^d/b,  & bq^d, &
  q^{d}/c, & q^{d-tn}, & q^{d+tn}
\end{array};q^d,\, \frac{q^{2d-3r}}{c}
\right].
 \label{eq:wei-bbb}
\end{align}
According to Watson's $_8\phi_7$ transformation (cf. \cite[Appendix
(III.18)]{Gasper}):
\begin{align*}
& _{8}\phi_{7}\!\left[\begin{array}{cccccccc}
a,& qa^{\frac{1}{2}},& -qa^{\frac{1}{2}}, & b,    & c,    & d,    & e,    & q^{-n} \\
  & a^{\frac{1}{2}}, & -a^{\frac{1}{2}},  & aq/b, & aq/c, & aq/d, & aq/e, & aq^{n+1}
\end{array};q,\, \frac{a^2q^{n+2}}{bcde}
\right] \\[5pt]
&\quad =\frac{(aq, aq/de;q)_{n}} {(aq/d, aq/e;q)_{n}}
\,{}_{4}\phi_{3}\!\left[\begin{array}{c}
aq/bc,\ d,\ e,\ q^{-n} \\
aq/b,\, aq/c,\, deq^{-n}/a
\end{array};q,\, q
\right],
\end{align*}
the right-hand side of \eqref{eq:wei-bbb} can be expressed as
\begin{align*}
[tn](cq^r)^{(r-tn)/d}\frac{(cq^{2r};q^d)_{(tn-r)/d}}{(q^d/c;q^d)_{(tn-1)/d}}
\sum_{k=0}^{(tn-r)/d}\frac{(q^{d-r},cq^r,q^{r+tn},q^{r-tn};q^d)_k}{(q^d,q^d/b,bq^{d},cq^{2r};q^d)_k}q^{dk}.
\end{align*}
Since $(1-aq^{tn})$ and $(a-q^{tn})$ are relatively prime
polynomials, we find the result: modulo $(1-aq^{tn})(a-q^{tn})$,
\begin{align}
&\sum_{k=0}^{T}[2dk+r]\frac{(aq^r,q^r/a,bq^r,q^r/b,cq^r,q^r;q^d)_k}{(q^d/a,aq^d,q^{d}/b,bq^d,q^d/c,q^d;q^d)_k}\bigg(\frac{q^{2d-3r}}{c}\bigg)^k
\notag\\[5pt]
&\quad\equiv\,[tn](cq^r)^{(r-tn)/d}\frac{(cq^{2r};q^d)_{(tn-r)/d}}{(q^d/c;q^d)_{(tn-1)/d}}
\sum_{k=0}^{(tn-r)/d}\frac{(aq^r,q^r/a,cq^r,q^{d-r};q^d)_k}{(bq^d,q^d/b,cq^{2r},q^d;q^d)_k}q^{dk}.
\label{eq:wei-ccc}
\end{align}
Interchanging the parameters $a$ and $b$ in \eqref{eq:wei-ccc}, we
are led to the relation: modulo $(1-bq^{tn})(b-q^{tn})$,
\begin{align}
&\sum_{k=0}^{T}[2dk+r]\frac{(aq^r,q^r/a,bq^r,q^r/b,cq^r,q^r;q^d)_k}{(q^d/a,aq^d,q^{d}/b,bq^d,q^d/c,q^d;q^d)_k}\bigg(\frac{q^{2d-3r}}{c}\bigg)^k
\notag\\[5pt]
&\quad\equiv\,[tn](cq^r)^{(r-tn)/d}\frac{(cq^{2r};q^d)_{(tn-r)/d}}{(q^d/c;q^d)_{(tn-1)/d}}
\sum_{k=0}^{(tn-r)/d}\frac{(bq^r,q^r/b,cq^r,q^{d-r};q^d)_k}{(aq^d,q^d/a,cq^{2r},q^d;q^d)_k}q^{dk}.
\label{eq:wei-ddd}
\end{align}

Employing \eqref{relation-aa} and \eqref{relation-bb} and the
Chinese remainder theorem for coprime polynomials, we can derive,
from \eqref{eq:wei-ccc} and \eqref{eq:wei-ddd}, the
 $q$-supercongruence \eqref{eq:wei-aaa}.
\end{proof}

Whereafter, we shall give the following parametric generalizations
of Theorems \ref{thm-d} and \ref{thm-e}.

\begin{thm}\label{thm-aaa}
Let $n$, $d$ be positive integers and $r$ an integer such that
 $d+n-dn\leq r\leq tn$, $\gcd(n,d)=1$, and $n\equiv
r\pmod{d}$. Then, modulo
$[n](1-aq^{n})(a-q^{n})(1-bq^{n})(b-q^{n})$,
\begin{align*}
&\sum_{k=0}^{M}[2dk+r]\frac{(aq^r,q^r/a,bq^r,q^r/b,cq^r,q^r;q^d)_k}{(q^d/a,aq^d,q^{d}/b,bq^d,q^d/c,q^d;q^d)_k}\bigg(\frac{q^{2d-3r}}{c}\bigg)^k
\notag\\[5pt]
&\quad\equiv\,[n](cq^r)^{(r-n)/d}\frac{(cq^{2r};q^d)_{(n-r)/d}}{(q^d/c;q^d)_{(n-r)/d}}
\notag\\[5pt]
&\quad\times\,\bigg\{\frac{(1-bq^{n})(b-q^{n})(-1-a^2+aq^{n})}{(a-b)(1-ab)}
\sum_{k=0}^{(n-r)/d}\frac{(aq^r,q^r/a,cq^r,q^{d-r};q^d)_k}{(bq^d,q^d/b,cq^{2r},q^d;q^d)_k}q^{dk}
\notag\\[5pt]
&\quad\quad+\:\frac{(1-aq^{n})(a-q^{n})(-1-b^2+bq^{n})}{(b-a)(1-ba)}
\sum_{k=0}^{(n-r)/d}\frac{(bq^r,q^r/b,cq^r,q^{d-r};q^d)_k}{(aq^d,q^d/a,cq^{2r},q^d;q^d)_k}q^{dk}\bigg\},
\end{align*}
where $M=(n-r)/d$ or $n-1$.
\end{thm}

\begin{proof}
Because $(1-aq^{n})(a-q^{n})(1-bq^{n})(b-q^{n})$ and $[n]$ are
relatively prime polynomials, we can prove Theorem \ref{thm-aaa}
through \eqref{eq:NW-a}, \eqref{eq:NW-b}, and the $t=1$ case of
Proposition \ref{prop-aaa}.
\end{proof}

\begin{thm}\label{thm-bbb}
Let $n$, $d$ be integers such that $n+r\geq d\geq3$, $\gcd(n,d)=1$,
and $n\equiv -r\pmod{d}$. Then, modulo
$[n]\Phi_n(q)(1-aq^{dn-n})(a-q^{dn-n})(1-bq^{dn-n})(b-q^{dn-n})$,
\begin{align*}
&\:\:\sum_{k=0}^{M}[2dk+r]\frac{(aq^r,q^r/a,bq^r,q^r/b;q^d)_k(q^r;q^d)_k^2}{(q^d/a,aq^d,q^{d}/b,bq^d;q^d)_k(q^d;q^d)_k^2}q^{(2d-3r)k}
\notag\\[5pt]
&\quad\equiv\,[dn-n]q^{r(r+n-dn)/d}\frac{(q^{2r};q^d)_{(dn-n-r)/d}}{(q^d;q^d)_{(dn-n-1)/d}}
\notag\\[5pt]
&\quad\times\,\bigg\{\frac{(1-bq^{dn-n})(b-q^{dn-n})(-1-a^2+aq^{dn-n})}{(a-b)(1-ab)}
\sum_{k=0}^{(dn-n-r)/d}\!\frac{(aq^r,q^r/a,q^r,q^{d-r};q^d)_k}{(bq^d,q^d/b,q^{2r},q^d;q^d)_k}q^{dk}
\notag\\[5pt]
&\quad\quad+\:\frac{(1-aq^{dn-n})(a-q^{dn-n})(-1-b^2+bq^{dn-n})}{(b-a)(1-ba)}
\sum_{k=0}^{(dn-n-r)/d}\!\frac{(bq^r,q^r/b,q^r,q^{d-r};q^d)_k}{(aq^d,q^d/a,q^{2r},q^d;q^d)_k}q^{dk}\bigg\},
\end{align*}
where $r=\pm1$ and $M=(dn-n-r)/d$ or $n-1$.
\end{thm}

\begin{proof}
Since $(1-aq^{dn-n})(a-q^{dn-n})(1-bq^{dn-n})(b-q^{dn-n})$ and
$[n]\Phi_n(q)$ are relatively prime polynomials, we can establish
Theorem \ref{thm-bbb} via \eqref{eq:NW} and the $c=1, r=\pm1, t=d-1$
case of Proposition \ref{prop-aaa}.
\end{proof}

Now we prepare to prove Theorems \ref{thm-d} and \ref{thm-e}.

\begin{proof}[Proof of Theorem \ref{thm-d}]
Letting $b\to1$ in Theorem \ref{thm-aaa}, we obtain the conclusion:
modulo $[n]\Phi_n(q)^2(1-aq^{n})(a-q^{n})$,
\begin{align}
&\sum_{k=0}^{M}[2dk+r]\frac{(aq^r,q^r/a,cq^r;q^d)_k(q^r;q^d)_k^3}{(q^d/a,aq^d,q^d/c;q^d)_k(q^d;q^d)_k^3}\bigg(\frac{q^{2d-3r}}{c}\bigg)^k
\notag\\[5pt]
&\quad\equiv\,[n](cq^r)^{(r-n)/d}\frac{(cq^{2r};q^d)_{(n-r)/d}}{(q^d/c;q^d)_{(n-r)/d}}
\notag\\[5pt]
&\quad\times\,\bigg\{(1-q^{n})^2
\sum_{k=0}^{(n-r)/d}\frac{(aq^r,q^r/a,cq^r,q^{d-r};q^d)_k}{(q^d;q^d)_k^3(cq^{2r};q^d)_k}q^{dk}
\notag\\[5pt]
&\quad\quad+\:\frac{a(1-q^{n})^2(2-q^{n})}{(1-a)^2}
\sum_{k=0}^{(n-r)/d}\frac{(aq^r,q^r/a,cq^r,q^{d-r};q^d)_k}{(q^d;q^d)_k^3(cq^{2r};q^d)_k}q^{dk}
\notag\\[5pt]
&\quad\quad-\:\frac{(1-aq^{n})(a-q^{n})(2-q^{n})}{(1-a)^2}
\sum_{k=0}^{(n-r)/d}\frac{(q^r;q^d)_k^2(cq^r,q^{d-r};q^d)_k}{(aq^d,q^d/a,cq^{2r},q^d;q^d)_k}q^{dk}\bigg\}.
\label{eq:wei-eee}
\end{align}
 By the L'H\^{o}pital rule, we have
\begin{align*}
&\lim_{a\to1}\bigg\{\frac{a(1-q^{n})^2}{(1-a)^2}
\sum_{k=0}^{(n-r)/d}\frac{(aq^r,q^r/a,cq^r,q^{d-r};q^d)_k}{(q^d;q^d)_k^3(cq^{2r};q^d)_k}q^{dk}\\[5pt]
&\qquad-\frac{(1-aq^{n})(a-q^{n})}{(1-a)^2}
\sum_{k=0}^{(n-r)/d}\frac{(q^r;q^d)_k^2(cq^r,q^{d-r};q^d)_k}{(aq^d,q^d/a,cq^{2r},q^d;q^d)_k}q^{dk}\bigg\}
\\[5pt]
&\:\:=\sum_{k=0}^{(n-r)/d}\frac{(q^r;q^d)_k^2(q^{d-r},cq^r;q^d)_k}{(q^d;q^d)_k^3(cq^{2r};q^d)_k}q^{dk}
\\[5pt]
&\qquad\times\bigg\{q^{n}-[n]^2\sum_{j=1}^{k}\bigg(\frac{q^{dj}}{[dj]^2}+\frac{q^{dj-d+r}}{[dj-d+r]^2}\bigg)\bigg\}.
\end{align*}
Letting $a\to1$ in \eqref{eq:wei-eee} and using the above limit, we
get Theorem \ref{thm-d}.
\end{proof}

\begin{proof}[Proof of Theorem \ref{thm-e}]
Letting $b\to1$ in Theorem \ref{thm-bbb}, we discover the formula:
modulo $[n]\Phi_n(q)^3 (1-aq^{dn-n})(a-q^{dn-n})$,
\begin{align}
&\sum_{k=0}^{M}[2dk+r]\frac{(aq^r,q^r/a;q^d)_k(q^r;q^d)_k^4}{(q^d/a,aq^d;q^d)_k(q^d;q^d)_k^4}q^{(2d-3r)k}
\notag\\[5pt]
&\quad\equiv\,[dn-n]q^{r(r+n-dn)/d}\frac{(q^{2r};q^d)_{(dn-n-r)/d}}{(q^d;q^d)_{(dn-n-r)/d}}
\notag\\[5pt]
&\quad\times\,\bigg\{(1-q^{dn-n})^2
\sum_{k=0}^{(dn-n-r)/d}\frac{(aq^r,q^r/a,q^r,q^{d-r};q^d)_k}{(q^d;q^d)_k^3(q^{2r};q^d)_k}q^{dk}
\notag\\[5pt]
&\quad\quad+\:\frac{a(1-q^{dn-n})^2(2-q^{dn-n})}{(1-a)^2}
\sum_{k=0}^{(dn-n-r)/d}\frac{(aq^r,q^r/a,q^r,q^{d-r};q^d)_k}{(q^d;q^d)_k^3(q^{2r};q^d)_k}q^{dk}
\notag\\[5pt]
&\quad\quad-\:\frac{(1-aq^{dn-n})(a-q^{dn-n})(2-q^{dn-n})}{(1-a)^2}
\sum_{k=0}^{(dn-n-r)/d}\frac{(q^r;q^d)_k^3(q^{d-r};q^d)_k}{(aq^d,q^d/a,q^{2r},q^d;q^d)_k}q^{dk}\bigg\}.
\label{eq:wei-fff}
\end{align}
 By the L'H\^{o}pital rule, we arrive at
\begin{align*}
&\lim_{a\to1}\bigg\{\frac{a(1-q^{dn-n})^2}{(1-a)^2}
\sum_{k=0}^{(dn-n-r)/d}\frac{(aq^r,q^r/a,q^r,q^{d-r};q^d)_k}{(q^d;q^d)_k^3(q^{2r};q^d)_k}q^{dk}
\\[5pt]
&\qquad-\frac{(1-aq^{dn-n})(a-q^{dn-n})}{(1-a)^2}
\sum_{k=0}^{(dn-n-r)/d}\frac{(q^r;q^d)_k^3(q^{d-r};q^d)_k}{(aq^d,q^d/a,q^{2r},q^d;q^d)_k}q^{dk}\bigg\}
\end{align*}
\begin{align*}
&\:\:=\sum_{k=0}^{(dn-n-r)/d}\frac{(q^r;q^d)_k^3(q^{d-r};q^d)_k}{(q^d;q^d)_k^3(q^{2r};q^d)_k}q^{dk}\\[5pt]
&\qquad\times\bigg\{q^{dn-n}-[dn-n]^2\sum_{j=1}^{k}\bigg(\frac{q^{dj}}{[dj]^2}+\frac{q^{dj-d+r}}{[dj-d+r]^2}\bigg)\bigg\}.
\end{align*}
Letting $a\to1$ in \eqref{eq:wei-fff} and utilizing the upper limit,
we are led to Theorem \ref{thm-e}.
\end{proof}


\end{document}